\documentclass[a4paper,11pt,reqno]{amsart}
\usepackage{amsmath}
\usepackage{amsfonts}
\usepackage{amssymb}
\usepackage{graphicx}
\usepackage{amsthm}
\usepackage{color}
\usepackage{url}
\usepackage{bigints}

\newtheorem{theorem}{Theorem}

\newtheorem{lemma}[theorem]{Lemma}

\newtheorem{remark}[theorem]{Remark}

\numberwithin{theorem}{section} 
\numberwithin{equation}{section} 

\makeatletter
\newcommand{\addresseshere}{%
  \enddoc@text\let\enddoc@text\relax
}
\makeatother

\usepackage[margin=1in]{geometry}

\begin{document}
\title[Nonexistence of radial optimal functions for the Sobolev inequality on $ \mathbb{M}^n $]{Nonexistence of radial optimal functions for \\ the Sobolev inequality on Cartan-Hadamard manifolds}

\author{Tatsuki Kawakami and Matteo Muratori}

\address{Tatsuki Kawakami, Department of Applied Mathematics and Informatics, Ryukoku University, 1-5 Yokotani, Seta Oe-cho, Otsu, Shiga 520-2194, Japan}
\email{kawakami@math.ryukoku.ac.jp}

\address{Matteo Muratori, Dipartimento di Matematica, Politecnico di Milano, Piazza Leonardo da Vinci 32, 20133 Milano, Italy}
\email{matteo.muratori@polimi.it}


\begin{abstract}
It is well known that the Euclidean Sobolev inequality holds on any Cartan-Hadamard manifold of dimension $ n\ge 3 $, i.e.~any complete, simply connected Riemannian manifold with nonpositive sectional curvature. As a byproduct of the Cartan-Hadamard conjecture, a longstanding problem in the mathematical literature settled only very recently in a breakthrough paper by Ghomi and Spruck \cite{GS}, we can now assert that the optimal constant is also Euclidean, namely~it coincides with the one achieved in the Euclidean space $ \mathbb{R}^n $ by the Aubin-Talenti functions. One may ask whether there exist at all optimal functions on a generic Cartan-Hadamard manifold $ \mathbb{M}^n $. What we prove here, with ad hoc arguments that do not take advantage of the validity of the Cartan-Hadamard conjecture, is that this is false at least for functions that are radially symmetric with respect to the geodesic distance from a fixed pole. More precisely, we show that if the optimum in the Sobolev inequality is achieved by some radial function, then $\mathbb{M}^n $ must be isometric to $ \mathbb{R}^n $.  
\end{abstract}

\maketitle

%


\section{Introduction} 

A Cartan-Hadamard manifold is a complete and simply connected Riemannian manifold $ \mathbb{M}^n $ with everywhere nonpositive sectional curvature. By the Cartan-Hadamard theorem, any such manifold turns out to be topologically equivalent to the Euclidean space $ \mathbb{R}^n $; more precisely, the exponential map centered at any point $ o \in \mathbb{M}^n $ is a diffeomorphism. We refer to Subsection \ref{nfrg} for an account on this and further basic properties of Cartan-Hadamard manifolds. From the functional point of view, a remarkable and by now well-established fact is the validity, on every such manifold of dimension $ n \ge 3 $, of the \emph{Euclidean Sobolev inequality}
\begin{equation}\label{eq-sob-euc}
\left\| f \right\|_{L^{2^\ast}\!(\mathbb{M}^n)} \leq  C \left\| \nabla f \right\|_{L^2(\mathbb{M}^n)} \quad \forall f \in C_{c}^1(\mathbb{M}^n) \, , \qquad 2^\ast := \frac{2n}{n-2} \, ,
\end{equation}
for some positive constant $ C>0 $. Here, with the term \emph{Euclidean}, we simply mean that the exponent appearing in the left-hand side of \eqref{eq-sob-euc} is exactly the same as the one corresponding to the case $ \mathbb{M}^n \equiv \mathbb{R}^n $. It is possible to establish \eqref{eq-sob-euc} through several techniques: see Subsection \ref{kernel} for an explicit proof and for references to other arguments available in the literature. As concerns the value of the \emph{optimal constant}, which will be denoted by $ \mathfrak C $, the situation is more complicated. Indeed, it had been an open question until very recently whether $ \mathfrak C $ coincides with the \emph{Euclidean best constant} $ C_{\mathrm{E}} $, namely the one achieved in $ \mathbb{R}^n $ by the celebrated Aubin-Talenti functions \cite{Aubin2, Talenti}. It is plain, due to the local Euclidean structure of $ \mathbb{M}^n $, that $ \mathfrak C $ cannot be smaller than $ C_{\mathrm{E}} $ (see Subsection \ref{radial-opt}). The fact that $ \mathfrak C = C_{\mathrm{E}} $ was known to be true up to dimension $ n=4 $, as a consequence of the validity of the so-called \emph{Cartan-Hadamard conjecture}, a longstanding problem in geometric analysis. The latter asserts that the \emph{isoperimetric inequality}, or equivalently the $1$-Sobolev inequality
\begin{equation}\label{1-sob}
\left\| f \right\|_{L^{1^\ast}\!(\mathbb{M}^n)} \leq  C_1 \left\| \nabla f \right\|_{L^1(\mathbb{M}^n)} \quad \forall f \in C_{c}^1(\mathbb{M}^n) \, , \qquad 1^\ast := \frac{n}{n-1} \, ,
\end{equation}
holds with \emph{Euclidean best constant} $C_1$ and the optimal functions are characteristic functions of Euclidean balls, i.e.~equality is achieved if and only if $ \mathbb{M}^n \equiv \mathbb{R}^n $ and $ f = \chi_{B_r} $, $r>0$, after a routine extension of \eqref{1-sob} to the $BV$ space. The validity of the Cartan-Hadamard conjecture for $n \ge 5$ was settled only in 2019 by M.~Ghomi and J.~Spruck, in the preprint paper \cite{GS}. Once $ C_1 $ in \eqref{1-sob} can be taken equal to the Euclidean isoperimetric constant, then a Schwarz-type symmetrization technique allows one to show that the same holds for \eqref{eq-sob-euc}, namely $ \mathfrak C = C_{\mathrm{E}} $. We refer the reader to \cite[Section 8]{Hebey} for an overview of the literature and the main techniques used until recently to attack the Cartan-Hadamard conjecture, along with its relation to $p$-Sobolev inequalities. 

\medskip

The aim of the present paper is to give a first contribution to the study of possible \emph{optimal functions}, i.e.~nontrivial functions attaining the identity in \eqref{eq-sob-euc} with $ C = \mathfrak C $. Indeed, regardless of the knowledge of the value of the optimal constant $ \mathfrak C $, it is reasonable to ask whether \eqref{eq-sob-euc} admits at all optimal functions and, in case of positive answer, what is the shape of the latter. We will work in the simplified \emph{radially-symmetric} framework, that is we will consider functions $ f(x) \equiv f(r(x)) $ that depend only on the geodesic distance $ r(x) :=\operatorname{d}(x,o) $ from a fixed pole $ o \in \mathbb{M}^n $, namely \emph{radial functions}. This may appear as a strong restriction, nevertheless radial symmetry has proved to play a major role in the investigation of extremal functions for a wide class of Sobolev-type inequalities. The literature here is huge: without any claim of completeness, in addition to the pioneering papers \cite{Aubin2, Talenti}, we quote \cite{DEL,DELM,DMN} and references therein for a thorough study of symmetry/symmetry-breaking issues in \emph{Caffarelli-Kohn-Nirenberg inequalities}, the latter being functional inequalities of the type of \eqref{eq-sob-euc} (possibly in interpolation form) with respect to power-type weights in $\mathbb{R}^n$. In fact \eqref{eq-sob-euc}, especially when restricted to radial functions, can be seen as a Euclidean \emph{weighted} inequality. See in particular Subsection \ref{lrs} below and \cite{MR}. 

\medskip 

Our main result is the following.
\begin{theorem}\label{teo-nonex}
	Let $ \mathbb{M}^n$ ($ n \ge 3 $) be a Cartan-Hadamard manifold. Suppose that the Sobolev inequality \eqref{eq-sob-euc} admits a (nontrivial) radial optimal function. Then $  \mathbb{M}^n $ is isometric to $ \mathbb{R}^n $.
\end{theorem} 
Clearly the above theorem can be interpreted both in terms of \emph{nonexistence} and in terms of \emph{rigidity}, in the sense that as soon as $ \mathbb{M}^n \not \equiv \mathbb{R}^n $ there exists no (radial) optimal function to \eqref{eq-sob-euc} and, should such a function exist, it is necessarily an Aubin-Talenti profile on $ \mathbb{M}^n \equiv \mathbb{R}^n  $. Note that optimal functions are naturally sought in $ \dot{H}^1(\mathbb{M}^n)$, i.e.~the closure of $  C^1_c(\mathbb{M}^n)$ with respect to the $ L^2(\mathbb{M}^n) $ norm of the gradient. We will provide three different proofs of Theorem \ref{teo-nonex} in Section \ref{proofs}. We want to emphasize that none of them takes advantage of the Cartan-Hadamard conjecture; all of our arguments only rely on classical Laplacian and volume-comparison tools (Subsection \ref{lrs}), along with the specific structure of the inequality in the radially-symmetric framework. It is possible, however, that a stronger version of Theorem \ref{teo-nonex} can be obtained upon \emph{assuming} the Cartan-Hadamard conjecture (see Remark \ref{rrr}). 

\medskip 

The investigation of optimal constants in functional inequalities has a long story. As we have already commented, the very first result dealing with the optimal functions for the Euclidean Sobolev inequality is due to two simultaneous and independent papers by T.~Aubin \cite{Aubin2} and G.~Talenti \cite{Talenti}. In a series of subsequent articles \cite{Aubin1,Aubin3, Aubin Li}, Aubin continued the analysis of Sobolev-type inequalities and optimality issues on Riemannian manifolds. Some improvements on \cite{Aubin2} were then achieved by \cite{HV,Hebey ter}. At the level of rigidity results, in \cite{L99,Car2} it is shown, upon assuming curvature or volume-growth bounds \emph{from below}, respectively, that a Riemannian manifold supporting the Sobolev inequality \eqref{eq-sob-euc} with Euclidean constant is isometric to $\mathbb{R}^n$.

\medskip

Concerning Poincar\'e inequalities, H.P.~McKean \cite{McKean} proved that, if on a Cartan-Hadamard manifold the sectional curvature is bounded from above by a negative constant $ -k $, then in addition to \eqref{eq-sob-euc} there holds 
\begin{equation}\label{eq-poin-intro} 
\left\| f \right\|_{L^2(\mathbb M^n)} \le \frac{2}{\sqrt{k} \left( n-1 \right) } \left\| \nabla f \right\|_{L^2(\mathbb M^n)} \qquad \forall f \in C^1_c(\mathbb M^n)  \, .
\end{equation}
This is equivalent to the fact that the infimum of the spectrum of (minus) the Laplace-Beltrami operator on $ \mathbb{M}^n $ is bounded from below by the constant $ k(N-1)^2/4 $, in other words $ -\Delta $ has an explicit \emph{spectral gap}. Moreover, such constant is sharp since it is attained on the \emph{hyperbolic space} $ \mathbb{H}^n $ of curvature $ -k $. Also the requirement on the ``nondegeneracy'' of the curvature is, in some sense, sharp. Indeed, in \cite{LZ} it was shown that, on any complete noncompact Riemannian manifold, the (essential) spectrum of $ -\Delta $ starts from zero as soon as the Ricci curvature vanishes at infinity. An alternative, and much simpler proof of \eqref{eq-poin-intro} was carried out in \cite{MR}, by means of one-dimensional techniques which are to some extent related to the arguments we develop in Section \ref{proofs}. Such paper deals with the validity of (radial) inequalities that interpolate between \eqref{eq-sob-euc} and \eqref{eq-poin-intro}, under (power-type) bounds from above on the sectional curvature of $ \mathbb{M}^n $. In the special, but significant case of the hyperbolic space, it is worth quoting the recent contributions \cite{BGG}, where the Poincar\'e inequality is established with \emph{optimal remainder terms} of Hardy type, and \cite{Nguyen}, where the author proves a remarkable inequality on $ \mathbb{H}^n $ yielding simultaneously the \emph{optimal} Sobolev and Poincar\'e constants. In wider geometric settings, Hardy-type inequalities were also addressed in \cite{Car1}, for a class of nonstandard weights. 

\medskip

Finally, we recall that the Sobolev inequality \eqref{eq-sob-euc}, along with related Gagliardo-Nirenberg and Poincar\'e inequalities, was successfully exploited to prove (sharp) $ L^1 $-$L^\infty$ smoothing effects for the \emph{porous medium equation} \cite{GM16}  and finite-time extinction estimates for the \emph{fast diffusion equation} \cite{BGV08} on Cartan-Hadamard manifolds, thus reinforcing the well-known connection between (nonlinear) diffusion equations and functional inequalities. In this regard, we also mention \cite{GSC}, where Faber-Krahn inequalities on Riemannian manifolds are investigated and consequent \emph{heat-kernel} bounds are established. 

\section{Preliminary material}\label{sect2} 

In the following, we will provide an overview of the essential notions and tools that one needs to know when dealing with Cartan-Hadamard manifolds (Subsections \ref{nfrg}, \ref{lrs} and \ref{lc}), along with some well-established results regarding the Sobolev inequality, of which however we believe it is worth giving a direct proof, since we try to be as much as possible self contained (see in particular Subsections \ref{kernel} and \ref{radial-opt}). 
 
\subsection{Basics of Cartan-Hadamard manifolds}\label{nfrg} 

We recall that a Cartan-Hadamard manifold is an $n$-dimensional Riemannian manifold $ \left( M , \mathfrak g \right) $ which is complete, simply connected and has everywhere nonpositive sectional curvature. This assumption entails a very strong topological (and geometric) consequence, due to the Cartan-Hadamard theorem (see e.g.~\cite[Theorem 1.10]{Lee} or \cite[Theorem II.6.2]{Chavel}): the cut-locus of \emph{any} point $ o \in M $ is empty, so that the exponential map $ T_o M \equiv \mathbb{R}^n \ni y \mapsto \exp_o y \in M $ is actually a global diffeomorphism and therefore $ M $ is in particular a \emph{manifold with a pole} (we refer to \cite{GreeneWu} for an excellent monograph on this class of manifolds). More than that: any point can play the role of a pole. 

Before proceeding further, let us fix some notations. The (standard) symbol $ T_o M $ stands for the tangent space of $ M $ at $ o \in M $, and we recall that $ \exp_o $ is the map that to any element $ y \in T_o M $ associates the point reached at time $t=1$ by the constant-speed geodesic that starts from $ o $ at $t=0$ with velocity $y$. In general the exponential map is well defined only for small $  y $, but as we have just seen on Cartan-Hadamard manifolds it is in fact global.

 We employ the symbol ``$ \equiv $'' instead of ``$=$'' for identities that should be understood up to suitable (implicit) transformations. In the case of Riemannian manifolds, by $  (M_1,\mathfrak g_1) \equiv (M_2,\mathfrak g_2 ) $ we mean that $ M_1 $ is isometric to $ M_2 $, i.e.~there exists a diffeomorphism from $ M_1 $ onto $ M_2 $ which is also an isometry with respect to $ \mathfrak g_1 $ and $ \mathfrak g_2 $. Finally, in order to lighten notations, an $n$-dimensional Cartan-Hadamard manifold is simply denoted by $ \mathbb{M}^n $ and $ \operatorname{d}(\cdot,\cdot) $ is the corresponding distance on $ \mathbb{M}^n $ induced by  its metric $ \mathfrak g $.

At the level of curvatures, we denote by $ \mathrm{Sect}(x) $ the sectional curvature at $ x \in M $ with respect to a generic $2$-plan in the tangent space $ T_x M$, whereas $ \mathrm{Sect}_o(x) $ stands for the sectional curvature with respect to any $ 2 $-plan in $ T_x M $ {containing} the radial direction, also known as \emph{radial sectional curvature}. Similarly, we denote by $ \mathrm{Ric}(x) $ the Ricci curvature at $ x \in M $ as a quadratic form, whereas the number $ \mathrm{Ric}_o(x) $ stands for the Ricci curvature evaluated in the radial direction, i.e.~the \emph{radial Ricci curvature}. 

In the sequel, $ o \in \mathbb{M}^n $ will tacitly be considered a \emph{fixed} reference point elected as a pole, unless otherwise specified. In view of what we have recalled above, it is possible to exploit \emph{radial coordinates} about $o$, namely to any $ x \in \mathbb{M}^n \setminus \{o\} $ one can associate in a unique way a couple $ (r,\theta) \in (0,\infty) \times \mathbb{S}^{n-1} $, where $ \mathbb{S}^{n-1}  $ represents the $(n-1)$-dimensional unit sphere endowed with the usual round metric. Note that $ r $ is the distance between $x$ and $o$, while $ \theta $ is the starting direction of the geodesic that connects $ o $ to $ x $. In this way, the metric $ \mathfrak g  $ of $ \mathbb{M}^n $ at $ x \equiv (r,\theta) $ can be written as follows:
\begin{equation}\label{ch-metric}
 \mathfrak g = \mathrm{d}r^2 + \left\langle \mathsf{A}(r,\theta) \, \mathrm{d} \theta , \mathrm{d} \theta \right\rangle_{\theta} ,
\end{equation}
for a suitable linear map $ \mathsf{A}(r,\theta) $ giving rise to a quadratic form in the tangent space of $ \mathbb{S}^{n-1} $ at $ \theta $. Here the symbol $ \langle \cdot , \cdot \rangle_{\theta} $ stands for the inner product of  such tangent space that induces the norm $ \|  \cdot \|_\theta $, and in \eqref{ch-metric} we identify an element of the tangent space of $ \mathbb{M}^n $ at $ x \equiv (r,\theta) $ with $ (\mathrm{d}r,\mathrm{d}\theta) $, where $\mathrm{d}r$ is an arbitrary real number that represents displacement in the radial direction and $\mathrm{d}\theta$ is an element of the tangent space of $ \mathbb{S}^{n-1} $ at $\theta$, that represents angular displacement. 

To our purposes, a key role is played by the positive scalar function 
$$ 
A(r,\theta) := \sqrt{\operatorname{det}\!\left[ \mathsf{A}(r,\theta) \right]} \qquad \forall (r,\theta) \in (0,\infty) \times \mathbb{S}^{n-1} \, .
$$
In fact $ A(r,\theta) $ coincides with the density of the volume measure of $ \mathbb{M}^n $, which we denote by $ d\mu $, with respect to the product measure $ dr \otimes d\theta $. Here and below, with some abuse of notation, the symbol $ {d}r $ stands for the Lebesgue measure on $ (0,\infty) $ and $ {d}\theta $ for the volume (i.e.~surface) measure of $ \mathbb{S}^{n-1}$, still endowed with the standard round metric. It is plain that,  since the metric of $ \mathbb{M}^n $ is locally Euclidean, or more rigorously $ \mathfrak g $ is differentiable on $ \mathbb{M}^n $, in particular there holds
\begin{equation}\label{ee1}
\lim_{r \downarrow 0} \frac{A(r,\theta)}{r^{n-1}} = 1 \qquad \text{uniformly w.r.t.~} \theta \in \mathbb{S}^{n-1} \, .
\end{equation}

Let us denote by $ B_r $ the geodesic ball of radius $r>0$, implicitly centered at $ o $, i.e.~the open set of points in $ \mathbb{M}^n $ whose distance from $ o $ is less than $r$. If the center of the ball is another point $ x \neq o $, we will write more explicitly $ B_r(x) $. Similarly, the boundary of $ B_r $, that is the geodesic sphere of all points at distance $r$ from $o$, is denoted by $  S_r $. Note that $S_r$ itself is an  $ (n-1) $-dimensional Riemannian manifold embedded in $ \mathbb{M}^n $. From the definition of $ A(r,\theta) $, we infer that for any fixed $ r>0 $ the function $ \theta \mapsto A(r,\theta) $ is the density, with respect to $ d\theta $, of the volume (i.e.~surface) measure $d\sigma$ of $ S_r $; as a result,
\begin{equation}\label{def-meas}
\sigma(S_{r})=\int_{\mathbb{S}^{n-1}}A(r,\theta)\, d\theta  \, .
\end{equation}

\subsection{Laplace-Beltrami operator, radial functions and Sobolev spaces}\label{lrs}

After the previous introductory section, we are in position to describe more precisely the functional setting in which we work. First of all, given a smooth function $f$ on $ \mathbb{M}^n $, the \emph{Laplace-Beltrami operator} (also \emph{Laplacian} for short) applied to $f$ reads, in radial coordinates (see \cite[Section 3]{Grigor'yan} or \cite[Section 2.2]{GMV}), 
\begin{equation}\label{lap-belt}
\Delta f=\dfrac{\partial^2 f}{\partial r^2} + \mathsf{m}(r,\theta) \, \dfrac{\partial f}{\partial r}+\Delta_{S_{r}} f \, ,
\end{equation}
where $\Delta_{S_{r}}$ represents the Laplace-Beltrami operator on the submanifold $S_{r}$ and 
\begin{equation}\label{lap-m}
\mathsf{m}(r,\theta) := \dfrac{\partial }{\partial r} \!\left[ \log A(r,\theta) \right] \qquad \forall x \equiv (r,\theta) \in (0,\infty) \times \mathbb{S}^{n-1} \, .
\end{equation} 
It is immediate to check that in fact $ \mathsf{m}(r,\theta) $ coincides with the \emph{Laplacian of the distance} function $ r \equiv r(x) := \mathrm{d}(x,o) $, which is of key importance in the analysis of partial differential equations on manifolds due to crucial comparison results (see the next section). Note that, upon integrating \eqref{lap-m} from a fixed $r_0>0$ to $ r > r_0 $, we obtain the identity
\begin{equation}\label{eq:int-m}
\int_{r_{0}}^r \mathsf{m}(s,\theta)\, ds=\log A(r,\theta) - \log A(r_0,\theta) \qquad \forall (r, \theta) \in (r_0,\infty) \times \mathbb{S}^{n-1} \, ,
\end{equation}
that is 
\begin{equation*}\label{eq:int-A}
A(r,\theta)=e^{\int_{r_0}^r \mathsf{m}(s,\theta)\, ds +c_{\theta}} \quad \forall (r, \theta) \in (r_0,\infty) \times \mathbb{S}^{n-1} \, , \qquad \text{where } c_{\theta} := \log A(r_0,\theta) \, .
\end{equation*}
Strictly related to the Laplacian is the \emph{gradient} operator, which for $ C^1(\mathbb{M}^n)$ functions reads (in radial coordinates)
$$
\nabla  f \equiv \left( \tfrac{\partial f}{\partial r} , \nabla_{S_r} f  \right)  \qquad \Longrightarrow \qquad \left| \nabla f \right|^2 = \left| \tfrac{\partial f}{\partial r} \right|^2 +  \left\| \nabla_{S_r} f \right\|_\theta^2  ,
$$
where $ \nabla_{S_r} $ is in turn the gradient operator of the submanifold $  S_r$. Clearly both $ \Delta_{S_r} $ and $ \nabla_{S_r} $ can explicitly be written in terms of $ \mathsf{A}(r,\theta) $, which we avoid since we will only deal with \emph{radial} functions, namely functions on $ \mathbb{M}^n $ that depend solely on the radial coordinate, i.e.~$ f(r,\theta) \equiv f(r) $. In this special case, we adopt the simplified notation $ \frac{\partial f}{\partial r} \equiv f'  $. 

Given a measurable function $ f  : \mathbb{M}^n \to \mathbb{R} $ and $ p \in [1,\infty) $, we define its $ L^p(\mathbb{M}^n) $ norm as 
$$
\left\| f \right\|_{L^p(\mathbb{M}^n)}^p := \int_{\mathbb{M}^n} \left| f \right|^p d\mu = \int_0^{\infty}\int_{\mathbb{S}^{n-1}} \left|f(r,\theta)\right|^p \, A(r,\theta) \, d\theta dr \, .
$$
Analogously, for a $ C^1(\mathbb{M}^n) $ function, the $ L^2(\mathbb{M}^n) $ norm of its gradient is defined as 
$$
\left\| \nabla f \right\|_{L^2(\mathbb{M}^n)}^2 := \int_{\mathbb{M}^n} \left| \nabla f \right|^2 d\mu = \int_0^{\infty}\int_{\mathbb{S}^{n-1}} \left(  \left| \tfrac{\partial f}{\partial r} \right|^2 +  \left\| \nabla_{S_r} f \right\|_\theta^2 \right) A(r,\theta) \, d\theta dr \, . 
$$ 
In particular, upon setting
\begin{equation}\label{def-psistar}
\psi_\star(r) := \left[ \frac{ \int_{\mathbb{S}^{n-1}} A(r,\theta) \, d\theta}{ \left| \mathbb{S}^{n-1}  \right| } \right]^{\frac{1}{n-1}} \qquad \forall r > 0 \, ,
\end{equation}
where $ \left| \mathbb{S}^{n-1} \right| $ is the total surface measure of the $ (n-1) $-dimensional unit sphere, we deduce that for a $ C^1(\mathbb{M}^n) $ {radial function} $f$ there hold 
\begin{equation}\label{rad-lp}
\left\| f \right\|_{L^p(\mathbb{M}^n)}^p = \left| \mathbb{S}^{n-1} \right| \int_0^{\infty} \left| f(r) \right|^p \psi_\star(r)^{n-1} dr 
\end{equation}
and
\begin{equation}\label{rad-grad}
\left\| \nabla f \right\|_{L^2(\mathbb{M}^n)}^2 = \left| \mathbb{S}^{n-1} \right| \int_0^{\infty} \left| f'(r) \right|^2 \psi_\star(r)^{n-1} dr \, .
\end{equation}
The reason for the notation $ \psi_\star $ in \eqref{def-psistar} will be clearer in the next subsection. 

Finally, we denote by $ \dot{H}^1(\mathbb{M}^n) $ the Sobolev space defined as the closure of $ C^1_c(\mathbb{M}^n) $ with respect to $ \| \nabla (\cdot) \|_{L^2(\mathbb{M}^n)} $, endowed with the latter norm. It is apparent that all the above formulas still hold for functions in $ \dot{H}^1(\mathbb{M}^n) $, up to interpreting partial derivatives in the weak sense. Clearly the Sobolev inequality \eqref{eq-sob-euc} extends to the whole $ \dot{H}^1(\mathbb{M}^n) $, and it is (a priori) in this space that optimal functions should be sought. 

\subsection{Model manifolds, Laplacian and volume comparison}\label{lc}

A \emph{model manifold} is an $n$-dimensional Riemannian manifold $ \left( M , \mathfrak g \right) $ with a pole $ o \in M $ whose metric can be written, with respect to the radial coordinates about $ o $, as (see \cite[Section 3.10]{Grig09})
\begin{equation*}\label{model metric}
 \mathfrak g = \mathrm{d}r^2 + \psi(r)^2 \left\| \mathrm{d}\theta \right\|^2 _{\theta} ,
\end{equation*}
where $\psi: [0,\infty ) \to  [0,\infty ) $ is a function belonging to the class
\begin{equation}\label{def-A} 
\mathcal{F} := \left\{ \psi\in C^{\infty}((0,\infty))\cap C^{1}([0,\infty)) : \ \psi(0)=0 \, , \ \psi(r)>0 \ \, \forall r>0 \, , \ \psi^\prime(0)=1 \right\} .
\end{equation}
In other words, it corresponds to the particular case of \eqref{ch-metric} when $ \mathsf{A}(r,\theta) $ is the identity times $ \psi(r)^2 $. Hence,  it follows that $ A(r,\theta) = \psi(r)^{n-1} $. For instance, the \emph{Euclidean space} $\mathbb{R}^n$ corresponds to $ \psi(r)=r $, while the \emph{hyperbolic space} $ \mathbb{H}^n $ corresponds to $ \psi(r)= \sinh r $. Note that, in general, a model manifold need not be Cartan-Hadamard: the latter property is equivalent to requiring that $ \psi $ is in addition convex. Outside the class of Cartan-Hadamard manifolds, we recover the \emph{unit sphere} $ \mathbb{S}^{n-1} $ with the choice $ \psi(r) =\sin r $, at least for $ r$ ranging in the bounded interval $ [0,\pi) $.  
  
Having introduced model manifolds, we can briefly recall some classical results that compare, in radial coordinates, the Laplacian of the distance function (w.r.t.~to a given pole $o $) of a Cartan-Hadamard manifold $ \mathbb{M}^n $ with the Laplacian of the distance function of the model manifold which equals the curvature bounds. More precisely, if 
\begin{equation}\label{comp-sect}
\mathrm{Sect}_o(x)\leq - \dfrac{\psi''(r)}{\psi(r)} \qquad  \forall (r,\theta) \equiv x \in \mathbb{M}^n \setminus\lbrace o\rbrace
\end{equation}
for some function $\psi\in\mathcal{F}$, then
\begin{equation}\label{comp-sect-2}
\mathsf{m}(r,\theta) \geq (n-1) \, \dfrac{\psi'(r)}{\psi(r)} \qquad \forall (r,\theta) \in (0,\infty) \times \mathbb{S}^{n-1} \, .
\end{equation}
Similarly, if 
\begin{equation}\label{ricci}
\mathrm{Ric}_o(x) \geq- (n-1) \, \dfrac{\psi''(r)}{\psi(r)} \qquad \forall (r,\theta) \equiv x \in \mathbb{M}^n \setminus \lbrace o \rbrace 
\end{equation}
for another function $\psi\in\mathcal{F}$, then
\begin{equation*}
\mathsf{m}(r,\theta)\leq (n-1) \,\dfrac{\psi'(r)}{\psi(r)} \qquad \forall  (r,\theta) \in (0,\infty) \times \mathbb{S}^{n-1} \, .
\end{equation*} 
We point out that the equality cases of the above inequalities do correspond to model manifolds, i.e.~the radial sectional curvature of a model manifold coincides with the right-hand side of \eqref{comp-sect}, and the same holds for the radial Ricci curvature in \eqref{ricci}. Moreover, the Laplacian of the distance function on a model manifold is also a radial function that equals the right-hand side of \eqref{comp-sect-2}. For further details, see e.g.~\cite[Section 2.2]{GMV} and references therein. Our entire focus here is on Cartan-Hadamard manifolds. We mention, however, that the above comparison results do hold in much more general Riemannian frameworks, up to a possible weak interpretation of the inequalities: we refer the reader to \cite[Sections 1.2.3 and 1.2.5]{MRS} (see also \cite[Section 2]{GreeneWu} or \cite[Section 15]{Grigor'yan}). 

Because a Cartan-Hadamard manifold has everywhere nonpositive sectional curvature, by applying \eqref{comp-sect} and \eqref{comp-sect-2} with the trivial choice $ \psi(r) = r $ we immediately deduce that
\begin{equation}\label{lap-euc}
\mathsf{m}(r,\theta) \ge \frac{n-1}{r} \qquad \forall (r,\theta) \in (0,\infty) \times \mathbb{S}^{n-1} \, .
\end{equation}
This simple inequality has a key consequence that will be crucial to our strategy, namely the fact that the volume measure of $ \mathbb{M}^n $ is larger than the Euclidean one:
\begin{equation}\label{main-ineq}
A(r,\theta) \ge r^{n-1} \qquad \forall (r,\theta) \in (0,\infty) \times \mathbb{S}^{n-1} \, .
\end{equation}
To establish \eqref{main-ineq} let us notice that, by virtue of \eqref{eq:int-m} and \eqref{lap-euc}, for every $ r_0>0 $ there holds
$$
\log\!\left( \frac{r^{n-1}}{r_0^{n-1}} \right) \le \log \! \left( \frac{A(r,\theta)}{A(r_0,\theta)} \right) \qquad \forall (r,\theta) \in (r_0,\infty) \times \mathbb{S}^{n-1} \, ,
$$
so that by taking exponentials and letting $ r_0 \downarrow 0 $, using \eqref{ee1}, we obtain \eqref{main-ineq}. 

We mention that \eqref{main-ineq} is the analogue, in the very special Cartan-Hadamard setting, of the celebrated Bishop-Gromov comparison theorem: see e.g.~\cite[Theorem 1.1]{Hebey} or \cite[Theorem 1.13]{MRS} for a more general statement. As a particular case of the latter, one deduces that the volume of geodesic balls of a Riemannian manifold with nonnegative \emph{Ricci} curvature is at most Euclidean. On Cartan-Hadamard manifolds, given the nonpositive \emph{sectional} curvature, we have the opposite inequality.

\subsection{A simple proof of the Sobolev inequality on Cartan-Hadamard manifolds}\label{kernel}

For completeness, we provide an elementary proof of the validity of the Sobolev inequality on any $ n $-dimensional ($ n \ge 3 $) Cartan-Hadamard manifold. This is by now a well-established result and, as a consequence of the Cartan-Hadamard conjecture recently settled by Ghomi and Spruck \cite{GS}, also the optimal constant is known to be Euclidean. A proof of the $1$-Sobolev inequality, from which the standard Sobolev inequality \eqref{eq-sob-euc} easily follows (see \cite[Lemma 8.1]{Hebey}), can be found e.g.~in \cite[Theorem 8.3]{Hebey}. 

Our argument goes as follows. Let $ \mathcal{K}(x,y,t) $ be the \emph{heat kernel} of $ \mathbb{M}^n $, namely the (minimal) solution to
\begin{equation}\label{eee}
\begin{cases}
\tfrac{\partial}{\partial t} \mathcal{K}(\cdot,y,\cdot) = \Delta \mathcal{K}(\cdot,y,\cdot) & \text{in } \mathbb{M}^n \times (0,+\infty) \, , \\
 \mathcal{K}(\cdot,y,0) = \delta_y & \text{in } \mathbb{M}^n \, ,
\end{cases}
\end{equation}
where $ \delta_y $ stands for the Dirac delta centered at a given but arbitrary $ y \in \mathbb{M}^n $. Let $  \mathcal{K}_{\mathrm{E}} $ denote the \emph{Euclidean heat kernel}, that is
\begin{equation*}\label{kernel-euc}
\mathcal{K}_{\mathrm{E}}(r,t) = \frac{e^{-\frac{r^2}{4t}}}{\left( 4\pi t \right)^{\frac{n}{2}}} \qquad \forall (r,t) \in [0,\infty) \times (0,+\infty) \, ,
\end{equation*}
which solves the analogue of \eqref{eee} in $ \mathbb{R}^n $ with $r$ replaced by $ |x-y| $. For each $ y \in \mathbb{M}^n $, the function $ \mathbb{M}^n \times (0,+\infty) \ni (x,t) \mapsto \mathcal{K}_{\mathrm{E}}(\operatorname{d}(x,y),t) $ turns out to be a \emph{supersolution} to \eqref{eee}. Indeed, it is plain that $ \frac{\partial }{\partial r} \mathcal{K}_{\mathrm{E}} \le 0 $; hence, from Laplacian comparison (recall \eqref{lap-belt} and \eqref{lap-euc}), we have:
$$
\frac{\partial }{\partial t} \mathcal{K}_{\mathrm{E}} = \frac{\partial^2 }{\partial r^2} \mathcal{K}_{\mathrm{E}} + \frac{n-1}{r} \frac{\partial }{\partial r} \mathcal{K}_{\mathrm{E}} \ge \frac{\partial^2 }{\partial r^2} \mathcal{K}_{\mathrm{E}} + \mathsf{m}(r,\theta) \frac{\partial }{\partial r} \mathcal{K}_{\mathrm{E}} \, .
$$
Upon setting $ r \equiv r(x) := \operatorname{d}(x,y) $, the above inequality is equivalent to the fact that $(x,t) \mapsto \mathcal{K}_{\mathrm{E}}(\operatorname{d}(x,y),t)$ is a supersolution to the differential equation in \eqref{eee}. On the other hand, because the volume measure of $ \mathbb{M}^n $ is locally Euclidean, i.e.~\eqref{ee1} holds, it is straightforward to check that  this function also attains a Dirac delta centered at $ y $ as $ t \downarrow 0 $. Hence, by the comparison principle and the arbitrariness of $y$, we infer that
\begin{equation}\label{eq-comp}
\mathcal{K}(x,y,t) \le \mathcal{K}_{\mathrm{E}}(\operatorname{d}(x,y),t) \le \frac{1}{\left( 4\pi t \right)^{\frac{n}{2}}} \qquad \forall (x,y,t) \in \mathbb{M}^n \times \mathbb{M}^n \times (0,+\infty) \, .
\end{equation}
As concerns the just mentioned comparison principle, we limit ourselves to observing that the latter can rigorously be established by both approximating $ \delta_y $ with a sequence of smooth radially decreasing data and filling $ \mathbb{M}^n $ with a sequence of geodesic balls centered at $y$, solving the analogues of \eqref{eee} with homogeneous Dirichlet boundary conditions.

Once \eqref{eq-comp} has been proved, \eqref{eq-sob-euc} is then a consequence of well-known equivalence results between pointwise heat-kernel bounds and the validity of Sobolev-type inequalities: see e.g.~\cite[Corollary 14.23]{Grig09} or \cite[Lemma 2.1.2 and Theorem 2.4.2]{Dav}. 

We point out that, in the above argument, the optimality of the constants might be lost in the passage from the bound \eqref{eq-comp} to \eqref{eq-sob-euc}. Hence to claim \eqref{eq-sob-euc} with \emph{Euclidean} constant  \eqref{eq-comp} is not enough, and we necessarily have to invoke the breakthrough result \cite[Theorem 1.1]{GS} combined with \cite[Proposition 8.2]{Hebey}. \hfill \qed 

\subsection{The optimal Sobolev constant is not smaller than the Euclidean one}\label{radial-opt}

The fact that the optimal constant $ \mathfrak C $ in the Sobolev inequality \eqref{eq-sob-euc} cannot be smaller than the Euclidean optimal constant $ C_{\mathrm{E}} $, which is attained in $ \mathbb{R}^n $ by the Aubin-Talenti functions (see \cite{Aubin2,Talenti})
\begin{equation}\label{AT}
f_{b}(x) \equiv f_{b}(|x|) := \left( 1 + b \left| x \right|^2 \right)^{-\frac{n-2}{2}} \quad \forall x \in \mathbb{R}^n \, , \qquad \text{where } b>0 \text{ is an arbitrary constant} ,
\end{equation}
is a plain consequence of the local Euclidean structure of $ \mathbb{M}^n $, and it is actually true on \emph{any} $n$-dimensional Riemannian manifold where \eqref{eq-sob-euc} holds. Note that in \eqref{AT} there should appear a further degree of freedom due to translations and another one due to multiplications, which we omit since it is inessential to our purposes (we only need scaling invariance). As observed in the Introduction, after \cite{GS} we can assert that in fact $ \mathfrak C = C_{\mathrm{E}} $. Nevertheless, because in Section \ref{proofs} we will only take advantage of the (crucial) inequality $ C_{\mathrm{E}} \le \mathfrak C $, we believe it is worth providing a direct (elementary and classical) proof. 

To this end, first of all note that, thanks to \eqref{ee1} and \eqref{main-ineq}, for every $ \varepsilon \in (0,1) $ there exists a positive constant $ c(\varepsilon) $ such that 
\begin{equation}\label{loc-vol}
r^{n-1} \le A(r,\theta) \le \left( 1+c(\varepsilon) \right) r^{n-1} \quad \forall (r,\theta) \in (0,\varepsilon) \times \mathbb{S}^{n-1}  \, , \qquad \lim_{\varepsilon \downarrow 0} c(\varepsilon) = 0 \, .
\end{equation}
We can therefore exploit \eqref{loc-vol} along with the explicit expression of the Aubin-Talenti functions. Let us consider the following ``truncated'' versions of \eqref{AT}: given $ \varepsilon \in (0,1) $, we set 
$$
f_{b,\varepsilon}(x) \equiv f_{b,\varepsilon}(|x|) := \left[ f_{b}(|x|) - f_{b}(\varepsilon) \right]^+ \qquad \forall x \in \mathbb{R}^n \, .
$$
It is readily seen that 
$$
\lim_{b \to \infty} \frac{\left\| f_{b,\varepsilon} \right\|_{L^{2^\ast}\!(\mathbb{R}^n)}}{\left\| \nabla f_{b,\varepsilon} \right\|_{L^2(\mathbb{R}^n)}}  = C_{\mathrm{E}} \qquad \forall \varepsilon \in (0,1) \, ,
$$
because {$ f_{b} $}, and hence also $ f_{b,\varepsilon} $, is concentrating at the origin as $ b \to  \infty $. Consider now the function $ g_{b,\varepsilon}(x) := f_{b,\varepsilon}(\operatorname{d}(x,o)) $, which belongs to $ \dot{H}^1(\mathbb{M}^n) $ and is supported by construction in $ \overline{B}_{\varepsilon} $. Thanks to \eqref{loc-vol} and the fact that $ g_{b,\varepsilon} $ is radial, recalling {\eqref{rad-lp} and \eqref{rad-grad}}, for every $ \varepsilon \in (0,1) $ there hold
$$
\begin{gathered}
\left\| f_{b,\varepsilon} \right\|_{L^{2^\ast}\!(\mathbb{R}^n)} \le \left\| g_{b,\varepsilon} \right\|_{L^{2^\ast}\!(\mathbb{M}^n)}  \le \left( 1+c(\varepsilon) \right)^{\frac{1}{2^\ast}} \left\| f_{b,\varepsilon} \right\|_{L^{2^\ast}\!(\mathbb{R}^n)} , \\
\left\| \nabla f_{b,\varepsilon} \right\|_{L^{2}(\mathbb{R}^n)} \le \left\| \nabla g_{b,\varepsilon} \right\|_{L^2(\mathbb{M}^n)}  \le \left( 1+c(\varepsilon) \right)^{\frac{1}{2}} \left\| \nabla f_{b,\varepsilon} \right\|_{L^{2}(\mathbb{R}^n)} .
\end{gathered}
$$
As a consequence, since the definition of $ \mathfrak C $ yields
$$
\frac{\left\| f_{b,\varepsilon} \right\|_{L^{2^\ast}\!(\mathbb{R}^n)} }{\left( 1+c(\varepsilon) \right)^{\frac{1}{2}} \left\| \nabla f_{b,\varepsilon} \right\|_{L^{2}(\mathbb{R}^n)}} \le  \frac{\left\| g_{b,\varepsilon} \right\|_{L^{2^\ast}\!(\mathbb{M}^n)}  }{\left\| \nabla g_{b,\varepsilon} \right\|_{L^2(\mathbb{M}^n)} } \le \mathfrak C \qquad \forall b>0 \, , \ \forall \varepsilon \in (0,1) \, ,
$$
by letting $ b \to\infty $ we infer that 
$$
\frac{C_{\mathrm{E}}}{\left( 1+c(\varepsilon) \right)^{\frac{1}{2}}} \le \mathfrak{C} \qquad \forall \varepsilon \in (0,1) \, ,
$$
whence the thesis upon letting $ \varepsilon \downarrow 0 $. \hfill \qed 

\section{The proof(s)}\label{proofs}

We provide three different proofs of Theorem \ref{teo-nonex}. The conclusion of each of them will be that the volume measure of $ \mathbb{M}^n $ is purely Euclidean, under the existence of an optimal radial profile for \eqref{eq-sob-euc}. For this reason, we first need a (rather intuitive) result ensuring that such property means that the Cartan-Hadamard manifold at hand is (isometric to) the Euclidean space. 

\begin{lemma}\label{isometry}
Let $ \mathbb{M}^n $ be a Cartan-Hadamard manifold. Suppose that its volume measure is Euclidean, that is
$$
A(r,\theta) = r^{n-1} \qquad \forall (r,\theta) \in (0,\infty) \times \mathbb{S}^{n-1} 
$$
with respect to radial coordinates about a fixed pole $ o \in \mathbb{M}^n $. Then $ \mathbb{M}^n \equiv \mathbb{R}^n $. 
\end{lemma}
\begin{proof}
We already know that the exponential map $ \mathbb{R}^n \ni y \mapsto  \exp_o y \in \mathbb{M}^n $ is a diffeomorphism, by the Cartan-Hadamard theorem (recall Subsection \ref{nfrg}). Let us show that it is also an isometry. Given any two points $ x_1 = \exp_o y_1 $ and $ x_2 = \exp_o y_2 $, because a Cartan-Hadamard manifold is a $ \mathrm{CAT}(0) $ space (see \cite[Theorem 1.3.3]{Bac} or \cite[Excercise IV.12]{Chavel}) there holds
\begin{equation}\label{cat}
\operatorname{d}(x_1,x_2) \ge \left| y_1-y_2 \right| ,
\end{equation}
i.e.~the length of the side of a geodesic triangle in $ \mathbb{M}^n $ opposite to the angle formed by the first two sides is not smaller than the length of side of the Euclidean triangle whose first two sides have the same length and angle. Our aim is to prove that \eqref{cat} is in fact an identity. Suppose by contradiction that there exist $ x_1 , \tilde{x}_2 \in \mathbb{M}^n $ such that
\begin{equation*}\label{cat-2}
r:= \operatorname{d}(x_1,\tilde{x}_2) > \left| y_1-\tilde{y}_2 \right| .
\end{equation*}
It is plain that \eqref{cat} yields 
\begin{equation*}\label{cat-3}
\left(\exp_o\right)^{-1}\!\left( B_r(x_1) \right) \subset B^{\mathrm{E}}_r(y_1) \, ,
\end{equation*}
where $ B^{\mathrm{E}}_r(y_1) $ stands for the Euclidean ball of radius $r$ centered at $ y_1 $. Hence, by continuity and the fact that the exponential map is a diffeomorphism, we deduce that actually there exists a nonempty open set $ \Omega \subset B^{\mathrm{E}}_r(y_1) $ such that
\begin{equation*}\label{cat-4}
\left(\exp_o\right)^{-1}\!\left( B_r(x_1) \right) \subset B^{\mathrm{E}}_r(y_1) \setminus \Omega \, .
\end{equation*}
Since, by assumption, the volume measure $ d\mu $ of $ \mathbb{M}^n $ is Euclidean, this would imply
$$
\mu\!\left(  B_r(y_1) \right) = \int_{\left(\exp_o\right)^{-1}\!\left( B_r(x_1) \right)} dy \le \int_{B^{\mathrm{E}}_r(y_1) \setminus \Omega} dy < \left| B^{\mathrm{E}}_r(y_1) \right| ,
$$
where $ dy $ denotes the $ n $-dimensional Lebesgue measure and $ \left| \cdot \right|  $ the corresponding volume of measurable sets. However, due to volume comparison (see \eqref{main-ineq} in Subsection \ref{lc}), this yields a contradiction since $ \mu\!\left(  B_r \right) \ge \left| B^{\mathrm{E}}_r \right| $ on any Cartan-Hadamard manifold, independently of the pole where $ B_r $ is centered.
\end{proof}

We are now in position to prove Theorem \ref{teo-nonex}.

\subsection{First proof: a weighted Euclidean inequality}\label{euc-weight}

The starting point consists of exploiting a suitable modification of the radial change of variables introduced in \cite[Section 7]{GMV} (see also \cite[Section 6]{VazHyp} in the case of the hyperbolic space). That is, let us set
$$
\frac{ds}{s^{n-1}} = \frac{dr}{\psi_\star(r)^{n-1}} \, ,
$$
or more precisely
\begin{equation}\label{chv}
\frac{1}{(n-2) s^{n-2}} = \int_{r}^{\infty}  \frac{dt}{\psi_\star(t)^{n-1}}  \, ,
\end{equation}
where $ \psi_\star $ is as in \eqref{def-psistar}. It is not difficult to check that $ \psi_\star $ belongs to the class $ \mathcal{F} $ defined in \eqref{def-A}. Moreover, $ \psi_\star' \ge 1  $ everywhere. Indeed, by combining \eqref{lap-m}, \eqref{lap-euc} and \eqref{main-ineq}, we have
\begin{equation}\label{psi-prime}
\psi_\star'(r) = \frac{\int_{\mathbb{S}^{n-1}} \frac{\partial}{\partial r} A(r,\theta) \, d\theta}{(n-1) \left| \mathbb{S}^{n-1}  \right|} \left[ \frac{ \int_{\mathbb{S}^{n-1}} A(r,\theta) \, d\theta}{ \left| \mathbb{S}^{n-1}  \right| } \right]^{\frac{1}{n-1}-1} \ge \frac{1}{r} \left[ \frac{ \int_{\mathbb{S}^{n-1}} A(r,\theta) \, d\theta}{ \left| \mathbb{S}^{n-1}  \right| } \right]^{\frac{1}{n-1}}\ge 1 \, .
\end{equation}
As a consequence, 
$$
\frac{1}{(n-2) s^{n-2}} \le \int_{r}^{\infty}  \frac{\psi_\star'(t)}{\psi_\star(t)^{n-1}} \, dt = \frac{1}{(n-2) \psi_\star(r)^{n-2}}  \, ,
$$
that is
\begin{equation}\label{we1}
\rho(s) := \frac{\psi_\star(r(s))}{s} \le 1 \qquad \forall s>0 \, .
\end{equation}
Let us write Rayleigh quotients in terms of the new variable $s$. To this end, given a (nontrivial) radial function $ f \equiv f(r) \in C^1_c(\mathbb{M}^n) $, we can construct another radial function $ \hat{f} \equiv \hat{f}(s)  := f(r(s)) \in C^1_c(\mathbb{R}^n) $, where $r(s)$ is obtained according to \eqref{chv}. It is plain that, for every $ p \in [1,\infty) $, the following identities hold (recall \eqref{rad-lp}): 
$$
\frac{ \left\|  f \right\|_{L^p(\mathbb{M}^n)}^p } {  \left| \mathbb{S}^{n-1} \right| } = \int_0^{\infty} \left| f(r) \right|^p \psi_\star(r)^{n-1} dr = \int_0^{\infty} \left| \hat f(s) \right|^p \rho(s)^{2(n-1)} s^{n-1} ds = \frac { \left\| \hat{f} \right\|_{L^p_\rho(\mathbb{R}^n)}^p } { \left| \mathbb{S}^{n-1} \right| } ,
$$
where for a function $ g \in L^p(\mathbb{R}^n) $ we set
$$
\left\| g \right\|_{L^p_\rho(\mathbb{R}^n)}^p := \int_{\mathbb{R}^n} \left| g(y) \right|^p \rho(|y|)^{2(n-1)} dy \, .
$$
Similarly (recall \eqref{rad-grad}), we have: 
$$
\begin{aligned}
\frac{ \left\| \nabla f \right\|_{L^2(\mathbb{M}^n)}^2 } {  \left| \mathbb{S}^{n-1} \right| } =  \int_0^{\infty} \left| f'(r) \right|^2 \psi_\star(r)^{n-1} dr = & \int_0^{\infty} \left| \hat{f}'(s) \, \frac{s^{n-1}}{\psi_\ast(r(s))^{n-1} } \right|^2 \frac{\psi_\ast(r(s))^{2(n-1)} }{s^{n-1}} \, ds \\
 = & \int_0^{\infty} \left| \hat{f}'(s)  \right|^2 s^{n-1}  ds \\
 = & \frac{ \left\| \nabla  \hat{f} \right\|_{L^2(\mathbb{R}^n)}^2 } {  \left| \mathbb{S}^{n-1} \right| } .
\end{aligned} 
$$
Hence, by virtue of \eqref{we1} and the Euclidean Sobolev inequality, we deduce that
\begin{equation}\label{sob-in}
\frac{\left\| \nabla f \right\|_{L^2(\mathbb{M}^n)}}{\left\|  f \right\|_{L^{2^\ast}\!(\mathbb{M}^n)}} = \frac{\left\| \nabla \hat f \right\|_{L^2(\mathbb{R}^n)}}{\left\|  \hat f \right\|_{L^{2^\ast}_\rho\!(\mathbb{R}^n)}} \ge  \frac{\left\| \nabla \hat f \right\|_{L^2(\mathbb{R}^n)}}{\left\|  \hat f \right\|_{L^{2^\ast}\!(\mathbb{R}^n)}}  \ge \frac{1}{C_{\mathrm{E}}} \, .
\end{equation}
Note that \eqref{sob-in} yields equivalence between the (radial) Sobolev inequality on $ \mathbb{M}^n $ and a (radial) weighted Euclidean Sobolev inequality. Clearly the latter can be extended to any nontrivial $ f  \in \dot{H}^1(\mathbb{M}^n) $ and therefore any nontrivial $ \hat{f} \in \dot{H}^1(\mathbb{R}^n) $, still in the radial framework. Suppose now that $ u \in \dot{H}^1(\mathbb{M}^n)  $ is a radial optimal function for the Sobolev inequality in $ \mathbb{M}^n $. Since we know from Subsection \ref{radial-opt} that the corresponding best constant cannot be smaller than the Euclidean one, from \eqref{sob-in} applied to $ f=u $ we deduce that in fact equality holds, whence
$$
\frac{\left\| \nabla \hat u \right\|_{L^2(\mathbb{R}^n)}}{\left\|  \hat u \right\|_{L^{2^\ast}\!(\mathbb{R}^n)}}  = \frac{1}{C_{\mathrm{E}}} \, .
$$
This means that $ \hat{u} $ is necessarily an Aubin-Talenti profile and 
$$
\left\|  \hat u \right\|_{L^{2^\ast}_\rho\!(\mathbb{R}^n)} = \left\|  \hat u \right\|_{L^{2^\ast}\!(\mathbb{R}^n)} \qquad \Longrightarrow \qquad \int_0^{\infty} \left| \hat{u}(s)\right|^{2^\ast} \! \left( 1-\rho(s) \right) ds = 0 \, .
$$
Because $ \hat u $ is everywhere positive (recall \eqref{AT}) and we know that $ \rho(s) \le 1 $ for all $ s > 0 $, we infer that $ \rho(s)=1 $ for all $s>0$; from the definition of $ \rho(s) $, there follows $ \psi_\star(r(s)) = s $ for all $ s>0 $. In view of \eqref{chv}, this identity can be rewritten as
$$
\psi_\star(r)^{n-2} = s(r)^{n-2} = \frac{1}{(n-2)  \int_{r}^{\infty}  \frac{dt}{\psi_\star(t)^{n-1}} } \qquad \forall r > 0 \, , 
$$
that is
$$
 \frac{d}{dr} \! \left( \int_{r}^{\infty}  \frac{dt}{\psi_\star(t)^{n-1}}  \right)  = - \left[ (n-2)  \int_{r}^{\infty}  \frac{dt}{\psi_\star(t)^{n-1}}  \right]^{\frac{n-1}{n-2}} \qquad \forall r > 0 \, , 
$$
which upon integration yields
$$
\int_{r}^{\infty}  \frac{dt}{\psi_\star(t)^{n-1}} = \frac{1}{(n-2) r^{n-2}} \qquad \forall r > 0 \, ,
$$
so that  $ s(r)=r $ and therefore $ \psi_\star(r) = r $ for all $r>0$. Because  $ A(r,\theta) \ge  r^{n-1} $ for all $ r>0 $ and $ \theta \in \mathbb{S}^{n-1} $, from the definition of $ \psi_\star $ we can finally deduce that in fact $ A(r,\theta) = r^{n-1} $, namely $ \mathbb{M}^n \equiv \mathbb{R}^n $ thanks to Lemma \ref{isometry}. \hfil \qed 

\subsection{Second proof: the Euler-Lagrange equation}\label{elproof}

First of all let us observe that, by classical variational arguments (see e.g.~\cite[Chapter I]{Struwe}), we can assume with no loss of generality that a radial optimal function is nonnegative and satisfies, up to a multiplication by a constant, the Euler-Lagrange equation
\begin{equation}\label{el1}
-\Delta u = -u'' - \mathsf{m}(r,\theta) \, u' = u^{2^\ast-1} \qquad \text{in } \mathbb{M}^n  \, ,
\end{equation}
where the spherical component $ \Delta_{S_r} $ of the Laplace-Beltrami operator in \eqref{lap-belt} has been neglected since $u$ is by assumption radial. Due to elliptic regularity (see again \cite[Appendix B]{Struwe}), we deduce that $u$ is at least $ C^{1,\alpha}(\mathbb{M}^n) $. Thanks to \eqref{lap-m}, note that  \eqref{el1} can be rewritten as 
\begin{equation}\label{el1-bis}
- \frac{1}{A(r,\theta)} \frac{\partial}{\partial r} \! \left( A(r,\theta) \, u' \right) = u^{2^\ast-1} \qquad \text{in } \mathbb{M}^n  \, ,
\end{equation}
which immediately implies that $ u $ is strictly radially decreasing, in particular it is everywhere strictly positive and therefore $C^\infty(\mathbb{M}^n)$ still by elliptic (bootstrap) regularity. Hence, recalling \eqref{lap-euc}, from \eqref{el1} there follows
\begin{equation}\label{el2}
-u'' - \frac{n-1}{r} \, u' \le u^{2^\ast-1} \qquad \forall r>0 \, .
\end{equation}
As in Subsection \ref{radial-opt} we have established that the optimal Sobolev constant $ \mathfrak C $ cannot be smaller than the Euclidean one $ C_{ \mathrm{E}} $, we have: 
\begin{equation}\label{el3}
\begin{aligned}
\left( \int_{0}^\infty \int_{\mathbb{S}^{n-1}}  u(r)^{2^\ast} A(r,\theta) \, d\theta dr \right)^{\frac{2}{2^\ast}} = & \, \mathfrak{C}^2 \int_{0}^\infty \int_{\mathbb{S}^{n-1}} \left|u'(r)\right|^2 A(r,\theta) \, d\theta dr \\
\ge & \, C_{ \mathrm{E}}^2 \int_{0}^\infty \int_{\mathbb{S}^{n-1}} \left|u'(r)\right|^2 A(r,\theta) \, d\theta dr \, .
\end{aligned}
\end{equation}
On the other hand, multiplying \eqref{el1-bis} by $  u A(r,\theta) $ and integrating, we obtain:  
$$
 \int_{0}^\infty \int_{\mathbb{S}^{n-1}} \left|u'(r)\right|^2 A(r,\theta) \, d\theta dr= \int_{0}^\infty \int_{\mathbb{S}^{n-1}} u(r)^{2^\ast} A(r,\theta) \, d\theta dr \, ,
$$
whence, in view of \eqref{el3}, 
\begin{equation}\label{el4}
\left(  \int_{0}^\infty \int_{\mathbb{S}^{n-1}} u(r)^{2^\ast} A(r,\theta) \, d\theta dr \right)^{\frac{2^\ast-2}{2^\ast}} \le \frac{1}{C_{ \mathrm{E}}^2}  \, .
\end{equation}
Since $ A(r,\theta) \ge r^{n-1} $, the radial profile $ u $, now interpreted as a function in $ \mathbb{R}^n $, is also an admissible competitor for the Euclidean Sobolev inequality, i.e.
\begin{equation}\label{el5}
\left( \int_{0}^\infty u(r)^{2^\ast} r^{n-1}  \left|\mathbb{S}^{n-1}\right|  dr \right)^{\frac{2}{2^\ast}}  \le  C_{\mathrm{E}}^2 \int_{0}^\infty \left|u'(r)\right|^2 r^{n-1}  \left|\mathbb{S}^{n-1}\right|  dr \, .
\end{equation}
By exploiting \eqref{el2} as above, we deduce that
\begin{equation}\label{el6}
\int_{0}^\infty  \left|u'(r)\right|^2 r^{n-1} \left|\mathbb{S}^{n-1}\right| dr \le \int_{0}^\infty  u(r)^{2^\ast} r^{n-1} \left|\mathbb{S}^{n-1}\right| dr \, .
\end{equation}
Hence, upon combining \eqref{el5} and \eqref{el6}, we end up with 
\begin{equation}\label{el7}
\frac{1}{C_{ \mathrm{E}}^2} \le \left( \int_0^\infty  u(r)^{2^\ast} r^{n-1} \left|\mathbb{S}^{n-1}\right| dr \right)^{\frac{2^\ast-2}{2^\ast}}  .
\end{equation}
Finally, \eqref{el4} and  \eqref{el7}  yield
$$
 \int_{0}^\infty \int_{\mathbb{S}^{n-1}} u(r)^{2^\ast} \left[ A(r,\theta) - r^{n-1} \right] d\theta dr  \le 0 \, .
$$
Since $ u $ is everywhere strictly positive and $A(r,\theta) \ge r^{n-1} $, this means that actually $ A(r,\theta) = r^{n-1} $, namely $ \mathbb{M}^n $ is isometric to the $ n $-dimensional Euclidean space due to Lemma \ref{isometry}.  \hfill \qed 

\subsection{Third proof: the (radial) isoperimetric inequality}

We borrow the main ideas from the proof \cite[Proposition 8.2]{Hebey}, also taking advantage of the fact that the functions we consider are purely radial. This approach is in some sense the dual of the one carried out in Subsection \ref{euc-weight}, where starting from the optimal function $u$ we constructed a Euclidean function $ \hat{u} $ preserving the $ L^2 $ norm of the gradient and increasing the $ L^{2^\ast} $ norm. Conversely, here we aim at constructing a Euclidean function that has the same $ L^{2^\ast} $ norm but lowers the $ L^2 $ norm of the gradient. To our purpose, let $ \Sigma , \Sigma_{\mathrm{E}} : ( 0,\infty ) \to (0,\infty) $ be defined as follows:
$$
\Sigma(v) := \int_{\mathbb{S}^{n-1}} A(R(v),\theta) \, d\theta  \, , \qquad \Sigma_{\mathrm{E}}(v) := \left| \mathbb{S}^{n-1} \right|^{\frac1n} \left( n v\right)^{\frac{n-1}{n}} \qquad \forall v>0 \, ,
$$
where $ v \mapsto R(v) $ is the inverse function of $ r \mapsto \mu(B_r) $. In other words, recalling formula \eqref{def-meas}, $ \Sigma(v) $ is the surface measure of the geodesic sphere on $ \mathbb{M}^n $ that encloses the geodesic ball of volume $v$, while $ \Sigma_{\mathrm{E}}(v) $ is the surface measure of the Euclidean sphere that encloses the Euclidean ball of volume $v$. It is not difficult to check that $ \Sigma(v) \ge \Sigma_{\mathrm{E}}(v) $ for all $ v>0 $, namely that the radial \emph{Euclidean isoperimetric inequality} holds in $ \mathbb{M}^n $. Indeed, this is equivalent to showing that
\begin{equation}\label{iso-radial}
\psi_\star(r) \ge \varrho(r)  \qquad \forall r>0 \, ,
\end{equation}
where $ \psi_\star $ is defined in \eqref{def-psistar} and $ r \mapsto \varrho(r) $ is the function that to any $r>0$ associates the radius of the Euclidean ball whose volume coincides with $ \mu(B_r) $. Such a function can easily be computed by imposing 
\begin{equation}\label{varrho}
\int _0^r \psi_\star(t)^{n-1} dt = \int_0^{\varrho(r)} t^{n-1}dt  \qquad  \Longrightarrow \qquad  \varrho(r) = \left( n \int _0^r \psi_\star(t)^{n-1} dt  \right)^{\frac1n}  \qquad \forall r>0 \, .
\end{equation}
Hence \eqref{iso-radial} does hold by virtue of the property $ \psi_\star' \ge 1 $ (recall \eqref{psi-prime}): 
$$
\varrho(r) = \left( n \int _0^r \psi_\star(t)^{n-1} dt  \right)^{\frac1n}  \le  \left( n \int _0^r \psi_\star(t)^{n-1} \, \psi_\star'(t) \, dt  \right)^{\frac1n}  = \psi_\star(r) \qquad \forall r > 0 \, .
$$
Now let us consider a nonnegative radial function $ f \equiv f(r) \in C^1(\mathbb{M}^n) $ and its corresponding transformed radial function $ \tilde{f} \equiv \tilde{f}(\varrho) \in C^1(\mathbb{R}^n) $ according to the following implicit relation:
\begin{equation}\label{rearr}
\mathcal{V}(\ell) :=  \mu\!\left( \left\{  f \ge \ell \right\} \right) = \left| \left\{  \tilde f \ge \ell \right\} \right|  \qquad \forall \ell > 0 \, ,
\end{equation}
where $ \left| \cdot \right| $ stands for the Euclidean volume function. Of course \eqref{rearr} does not determine $ \tilde{f} $ in a unique way unless $ \tilde{f} $ is additionally required to be radially decreasing, which gives rise to an analogue of the well-established \emph{Schwarz symmetrization}, originally exploited by Talenti  \cite{Talenti}. Note that, by construction, the functions $ f $ and $ \tilde{f} $ share the same $ L^p $ norms (possibly infinite). Indeed, for any $ p \in[1,\infty) $, by Fubini's theorem and \eqref{rearr} we have: 
\begin{equation}\label{equiv-norm}
\begin{aligned}
\left\| f \right\|_{L^p(\mathbb{M}^n)}^p = \int_{\mathbb{M}^n} f^p \, d\mu = \frac{1}{p} \int_{\mathbb{M}^n} \left( \int_0^f \ell^{p-1} d\ell \right) d\mu 
 = & \frac{1}{p} \int_0^{\infty} \ell^{p-1} \left( \int_{ f \ge \ell } d\mu \right) d\ell \\
 = &  \frac{1}{p} \int_0^{\infty} \ell^{p-1} \, \mathcal{V}(\ell) \, d\ell \\
 = & \left\| \tilde{f} \right\|_{L^p(\mathbb{R}^n)}^p .
\end{aligned}
\end{equation}
Let us deal with gradients (i.e.~radial derivatives). In the sequel, we additionally require that $ f'(r)<0 $ for all $ r>0 $, $ \tilde{f}'(\varrho)<0 $ for all $ \varrho >0 $ and $ \inf f = 0 $, so that in particular $ f $ and $ \tilde{f} $ are strictly radially decreasing (therefore everywhere positive) and vanish at infinity.  Note that, under such assumptions, there holds $ f(r) = \tilde{f}(\varrho(r)) $, where $ \varrho(r) $ is given in \eqref{varrho}. In this case it is easy to check that $ \ell \mapsto \mathcal{V}(\ell) $ is also a $ C^1((0,c)) $ function with $ \mathcal{V}'(\ell)<0 $ for all $ \ell \in (0,c) $, $c>0$ being the maximum of $ f $. Moreover, we have the following identities: 
\begin{equation}\label{gradient}
 f'\!\left( f^{-1}(\ell) \right) = \frac{\Sigma\!\left( \mathcal{V}(\ell) \right)}{\mathcal{V}'(\ell)} \qquad \text{and} \qquad  \tilde{f}'\!\left( \tilde{f}^{-1}(\ell) \right) = \frac{\Sigma_{\mathrm{E}}\!\left( \mathcal{V}(\ell) \right)}{\mathcal{V}'(\ell)}  \qquad \forall \ell \in (0,c) \, .
\end{equation} 
In fact \eqref{gradient} is a simple consequence of the (radial) \emph{co-area} formula 
\begin{equation}\label{coarea}
\int_{\mathbb{M}^n} g \, d\mu = \int_0^{\infty} g(r) \, \psi_\star(r)^{n-1} dr = -	\int_0^{c} \frac{ g\!\left( f^{-1}(\ell) \right) \psi_\star(R(\mathcal{V}(\ell)))^{n-1} } { {f}'\!\left( {f}^{-1}(\ell) \right) } \, d\ell \, ,
\end{equation}
valid for any measurable radial function $ g \ge 0 $, with the particular choice $ g =  \chi_{\{ f \ge z \}} $ for each level $ z \in (0,c) $. Clearly the same holds for $ f \equiv \tilde{f} $ and $ \mathbb{M}^n \equiv \mathbb{R}^n $. We point out that an analogue of \eqref{coarea} is available for a wide class of nonradial functions and general manifolds: see \cite[Section 8.2]{Hebey} and \cite[Chapter III]{Chavel}. However, in our simplified setting it follows directly from the change of variables $ r=f^{-1}(\ell) $ inside the integral. 

At this stage we are in position to conclude the proof. Indeed, if a (nonnegative) radial optimal function $ u \equiv u(r) \in \dot{H}^1(\mathbb{M}^n) $ exists, by virtue of the Euler-Lagrange equation \eqref{el1} we know that it is smooth, positive and satisfies $ u'(r)<0 $ for all $ r>0 $ (see the beginning of the proof in Subsection \ref{elproof}). By choosing $ f=u $ and $ g=\left| \nabla u \right|^2 = |u'|^2 $ in \eqref{coarea}, using \eqref{gradient}, we obtain: 
$$
\left\| \nabla  u \right\|_{L^2(\mathbb{M}^n)}^2 = - \int_0^{c} \frac{\Sigma(\mathcal{V}(\ell))^2}{\mathcal{V}'(\ell)} \, d\ell \ge - \int_0^{c} \frac{\Sigma_{\mathrm{E}}(\mathcal{V}(\ell))^2}{\mathcal{V}'(\ell)} \, d\ell = \left\| \nabla  \tilde u \right\|_{L^2(\mathbb{R}^n)}^2 , 
$$
where in the last passage we have exploited the radial isoperimetric inequality established in the beginning along with \eqref{gradient} and \eqref{coarea} also applied to $ \tilde{f} \equiv \tilde{u} $ and $ \mathbb{M}^n \equiv \mathbb{R}^n $.  On the other hand, if $ u $ is optimal we know that 
$$ 
\left\| \nabla  u \right\|_{L^2(\mathbb{M}^n)} \le \frac{ \left\| u \right\|_{L^{2^\ast}\!(\mathbb{M}^n)}}{C_{\mathrm{E}}} \overset{\eqref{equiv-norm}}{=}  \frac{ \left\| \tilde{u} \right\|_{L^{2^\ast}\!(\mathbb{R}^n)}}{C_{\mathrm{E}}} \le \left\| \nabla \tilde u \right\|_{L^2(\mathbb{R}^n)} .  
$$ 
Hence, by combining the last two formulas we end up with the identity
$$
 -\int_0^{c} \frac{\Sigma(\mathcal{V}(\ell))-\Sigma_{\mathrm{E}}(\mathcal{V}(\ell))}{\mathcal{V}'(\ell)} \, d\ell = 0 \, ,
$$
which yields $ \Sigma(\mathcal{V}(\ell)) = \Sigma_{\mathrm{E}}(\mathcal{V}(\ell)) $ for every $ \ell \in (0,c) $ since $ \Sigma \ge \Sigma_{\mathrm{E}} $, and it is readily seen that this implies $ \psi_\star(r)=r $ for all $ r>0 $, which proves the thesis in view of Lemma \ref{isometry}. \hfil \qed 

\begin{remark}\rm	\label{rrr}
We stress that each of the proofs of our main result, Theorem \ref{teo-nonex}, is completely independent of the knowledge of the optimal constant in \eqref{eq-sob-euc}, hence of the validity of the Cartan-Hadamard conjecture, which as mentioned above is a very delicate problem that has only recently been solved by \cite{GS}. However, it is reasonable that, in the light of the latter, a careful analysis of the proof of \cite[Proposition 8.2]{Hebey} yields the analogue of Theorem \ref{teo-nonex} even for \emph{nonradial} optimal functions.
\end{remark}

\medskip 

\noindent{\textbf{Acknowledgments.}} 
M.M. thanks the ``Gruppo Nazionale per l'Analisi Matematica, la Probabilit\`a e le loro Applicazioni'' (GNAMPA) of the ``Istituto Nazionale di Alta Matematica'' (INdAM, Italy). M.M.~was supported  by the GNAMPA Project 2018 ``Analytic and Geometric Problems Associated to Nonlinear Elliptic and Parabolic PDEs'', by the GNAMPA  Project 2019 ``Existence and Qualitative Properties for Solutions of Nonlinear Elliptic and Parabolic PDEs'' and by the PRIN Project 2017 ``Direct and Inverse Problems for Partial Differential Equations: Theoretical Aspects and Applications'' (MIUR, Italy). T.K.~was supported by JSPS KAKENHI Grant Numbers JP 19H05599 and 16K17629 (Japan).

M.M. is grateful to the Ryukoku University for the hospitality during his visit in April 2019, when part of this project was carried out.

\end{document}